\documentclass[11pt]{article}
\usepackage{tikz}
\usepackage{hyperref}
\usepackage{empheq}
\usepackage[shortlabels]{enumitem}
\setlist[enumerate]{nosep}
\usepackage[doc]{optional}
\usepackage{xcolor}
\usepackage{float}
\usepackage{soul}
\usepackage{graphicx}
\definecolor{labelkey}{rgb}{0,0.08,0.45}
\definecolor{refkey}{rgb}{0,0.6,0.0}
\definecolor{Brown}{rgb}{0.45,0.0,0.05}
\definecolor{lime}{rgb}{0.00,0.8,0.0}
\definecolor{lblue}{rgb}{0.5,0.5,0.99}

 \usepackage{mathpazo}

 \usepackage{stmaryrd}
\usepackage{amssymb}
    \newcommand*{\stardiff}{%
      \mathrel{\vcenter{\offinterlineskip
      \hbox{$\kern+2.0pt {}^\ast$}\vskip-1.20ex\hbox{$-$}}}}

\newcommand{\sepp}{\setlength{\itemsep}{-2pt}}

\newcommand{\preceqz}{\preceq_{\mathsf{Z}}}
\newcommand{\preceql}{\preceq_{\mathsf{L}}}
\newcommand{\preceqm}{\preceq_{\mathsf{M}}}

\newcommand{\frechet}{Fr{\'{e}}chet}

\newcommand{\nnn}{\ensuremath{{n\in{\mathbb N}}}}
\newcommand{\thalb}{\ensuremath{\tfrac{1}{2}}}
\newcommand{\menge}[2]{\big\{{#1}~\big |~{#2}\big\}}

\newcommand{\To}{\ensuremath{\rightrightarrows}}

\newcommand{\fenv}[1]%
{\ensuremath{\,\overrightarrow{\operatorname{env}}_{#1}}}
\newcommand{\benv}[1]%
{\ensuremath{\,\overleftarrow{\operatorname{env}}_{#1}}}

\newcommand{\infconv}{\ensuremath{\mbox{\small$\,\square\,$}}}

\newcommand{\scal}[2]{\left\langle{#1},{#2}  \right\rangle}

\newcommand{\exi}{\ensuremath{\exists\,}}

\newcommand{\RR}{\ensuremath{\mathbb R}}
\newcommand{\RP}{\ensuremath{\mathbb{R}_+}}
\newcommand{\RPP}{\ensuremath{\mathbb{R}_{++}}}

\newcommand{\dom}{\ensuremath{\operatorname{dom}}}
\newcommand{\qq}{\ensuremath{\operatorname{q}}}

\newcommand{\Id}{\ensuremath{\operatorname{Id}}}

\newcommand{\pinf}{\ensuremath{+\infty}}

{\begin{list}{}{%
\settowidth{\labelwidth}{\textrm{#1~}}%
\setlength{\leftmargin}{\labelwidth+\labelsep}}}%requires macro calc.sty
{\end{list}}
\usepackage{amsthm}
\usepackage[capitalize]{cleveref}
\crefname{equation}{}{equations}
\crefname{chapter}{Appendix}{chapters}
\crefname{item}{}{items}
\newtheorem{theorem}{Theorem}[section]
\newtheorem{lemma}[theorem]{Lemma}

\newtheorem{corollary}[theorem]{Corollary}

\newtheorem{definition}[theorem]{Definition}

\newtheorem{example}[theorem]{Example}
\newtheorem{ex}[theorem]{Example}
\newtheorem{fact}[theorem]{Fact}
\newtheorem{remark}[theorem]{Remark}

\providecommand{\nc}[1]{\operatorname{N}_{#1}}
\providecommand{\px}[1]{\operatorname{P}_{#1}}
\providecommand{\pj}[1]{\operatorname{P}_{#1}}

\providecommand{\jx}[1]{\operatorname{J}_{#1}}

\providecommand{\sd}[1]{\partial {#1}}
\providecommand{\env}[1]{\operatorname{env}({#1})}

\providecommand{\Gx}{\Gamma_0(X)}
\providecommand{\Fm}{\operatorname{\mathcal{F}}(X)} % firmly nonexpansive mappings on $X$.
\providecommand{\Pm}{\operatorname{\mathcal{P}}(X)} % firmly nonexpansive mappings on $X$.
\providecommand{\Mm}{\operatorname{\mathcal{M}}(X)} % firmly nonexpansive mappings on $X$.

\providecommand{\Ex}{\operatorname{M}_{\,0}(X)} % firmly nonexpansive mappings on $X$.

\providecommand{\norm}[1]{\lVert#1\rVert}
\providecommand{\normsq}[1]{\lVert#1\rVert^2}

\providecommand{\innp}[1]{\langle#1\rangle}

\providecommand{\RR}{\mathbb{R}}

\providecommand{\dom}{\operatorname{dom}}

\providecommand{\gr}{\operatorname{gra}}

\providecommand{\Id}{\operatorname{{ Id}}}

\providecommand{\To}{\rightrightarrows}

\providecommand{\gr}{\operatorname{gra}}

\providecommand{\Id}{\operatorname{Id}}

\providecommand{\RR}{\mathbb{R}}

\providecommand{\SR}{\mathbb{S}}

\definecolor{myblue}{rgb}{.8, .8, 1}

\allowdisplaybreaks % or locally if problems {\allowdisplaybreaks
\begin{document}

\title{\textsc
The resolvent order: a unification of the orders\\
by Zarantonello, by Loewner, and by Moreau}

\author{
Sedi Bartz\thanks{
Mathematics, University
of British Columbia,
Kelowna, B.C.\ V1V~1V7, Canada. E-mail:
\texttt{sedi.bartz@ubc.ca}.},~
Heinz H.\ Bauschke\thanks{
Mathematics, University
of British Columbia,
Kelowna, B.C.\ V1V~1V7, Canada. E-mail:
\texttt{heinz.bauschke@ubc.ca}.}
~~and Xianfu Wang\thanks{
Mathematics, University of British Columbia, Kelowna, B.C.\
V1V~1V7, Canada.
E-mail: \texttt{shawn.wang@ubc.ca}.}}

\date{June 28, 2016}

\maketitle

\begin{abstract}
\noindent
We introduce and investigate the resolvent order, which is 
a binary relation on the set of firmly nonexpansive
mappings. It unifies
well-known orders introduced by Loewner (for positive semidefinite
matrices) and by Zarantonello (for projectors onto convex
cones). 
A connection with Moreau's order of convex
functions is also presented.
We also construct partial orders on (quotient sets of) proximal
mappings and convex functions. 
Various examples illustrate our results. 
\end{abstract}
{\small
\noindent
{\bfseries 2010 Mathematics Subject Classification:}
{Primary 
06A06, %partial order
26B25, %real functions, convexity
47H05, %Monotone operators and generalizations
47H09, %Contraction-type mappings, nonexpansive mappings, $A$-proper mappings,
Secondary 
15B57, %hermitian matrices
47L07, %convex sets and cones of operators
52A41, %convex functions
90C25. %convex optimization
}

\noindent {\bfseries Keywords:}
Baillon--Haddad theorem,
convex cone, 
convex function, 
envelope,
firmly nonexpansive mapping,
Loewner order, 
L\"owner order,
maximally monotone operator,
Moreau envelope, 
Moreau's order,
positive semidefinite matrix, 
projection mapping,
projector, 
proximal mapping,
proximity operator,
resolvent,
resolvent order,
Zarantonello's order.
}

\section{Introduction}

In this paper, we assume that
%\begin{empheq}[box=\mybluebox]{equation}
\begin{equation}
\label{T:assmp}
X \text{~~is a real Hilbert space},
%\end{empheq} 
\end{equation}
with inner product $\innp{\cdot,\cdot}$ and
induced norm $\norm{\cdot}$. 
We denote the set of all functions from $X$ to $\RR\cup\{\pinf\}$
that are \emph{convex}\footnote{We assume the reader is familiar with
basic convex analysis; see, e.g., 
\cite{HU1,HU2,Rock70,Zalinescu,Rock98,BC2011}.}, lower semicontinuous and proper 
by  $\Gx$. 
Let $A\colon X\To X$ be a set-valued operator, i.e.,
$(\forall x\in X)$ $Ax\subseteq X$ and denote the graph of $A$ by
$\gr A$. 
Recall that $A$ is monotone if
\begin{equation}
(\forall (x,x^*)\in\gr A)(\forall (y,y^*)\in\gr A)\quad
\scal{x-y}{x^*-y^*}\geq 0
\end{equation}
and that $A$ is \emph{maximally monotone} if it is monotone and cannot
be extended without destroying monotonicity. 
The notion of maximal monotonicity has 
proven to be useful in modern optimization and nonlinear
analysis; see, e.g.,
\cite{BC2011,Borwein50,Brezis,BurIus,Rock98,Simons1,Simons2,Zeidler2a,Zeidler2b}. 
We denote the set of maximally monotone operators on $X$ by
$\Mm$. 
This set includes subdifferential operators of functions in $\Gx$ as
well as all square matrices with symmetric parts that are
positive semidefinite. 
Furthermore, we denote by $\Fm$ the set of all
mappings $T\colon X\to X$ that are 
\emph{firmly nonexpansive}\footnote{
Note that by the Cauchy--Schwarz inequality,
every firmly nonexpansive mapping is \emph{nonexpansive}, i.e.,
Lipschitz continuous with constant $1$.}, i.e., 
\begin{equation}
(\forall x\in X)(\forall y\in X)\quad
\|Tx-Ty\|^2 \leq \scal{x-y}{Tx-Ty}.
\end{equation}
Thanks to the work of Minty \cite{Minty} (see also \cite{EckBer}), 
we can identify a maximally monotone operator $A$ from $\Mm$ with
with its \emph{resolvent} in $\Fm$ via
\begin{equation}
\jx{A} := (\Id+A)^{-1}.
\end{equation}
Here $\Id=\nabla \qq = \partial \qq$ is the identity operator on
$X$, where $\qq \colon x\mapsto \tfrac{1}{2}\|x\|^2$. 
If we focus instead on the important subset of subdifferential
operators in $\Mm$, then we recover Moreau's \cite{Moreau} 
\emph{proximal mapping} (or proximity operator)
\begin{equation}
\px{f} := \jx{\sd{f}} = (\Id+\sd{f})^{-1},
\end{equation}
where $f\in\Gx$ and 
$\sd{f}\in\Mm$ is the subdifferential operator of $f$. 
The set of proximal mappings, which we write as $\Pm$,
can also be described as follows.
Given $f\in\Gx$, let $\env{f} := \qq\infconv f$ be the
\emph{(Moreau) envelope} of $f$, where $\infconv$ denotes infimal
convolution. The set of all envelopes is written as $\Ex$. 
Then 
\begin{equation}
\nabla \env{f^*} = \px{f} = \jx{\sd{f}}=(\Id+\sd{f})^{-1},
\end{equation}
where $f^*\in\Gx$ is the Fenchel conjugate of $f$. 
(Thus, we can loosely write $\Ex = \Gx \infconv \qq$ and $\nabla
\Ex = \Pm$.)

Having set up the necessary notation, we can now describe the
goal and the organization of this paper.

\emph{The goal of this paper is to introduce a new
order\footnote{To keep the language in this paper from being overly
technical, we will refer to an ``order'' as a binary relation that is
at least reflexive.} on $\Fm$ which we
call the \emph{resolvent order}.
It induces orders on $\Pm$, $\Mm$, $\Gx$, and $\Ex$ which will
allow us to 
unify and connect to
several well known orders from linear and nonlinear analysis,
namely to the orders by Zarantonello, by Loewner, and by Moreau.
We provide several examples and also present a partial order 
on (a quotient set of) the set of proximal mappings $\Pm$ and on
(a quotient set of) the set of convex functions $\Gx$. 
}

The remainder of the paper is organized as follows.
In Section~\ref{s:aux}, we present various results that
make the proofs of the main results more structured. 
The resolvent order on $\Fm$ is defined in Section~\ref{s:main},
where we also provide basic properties and characterizations 
for $\Pm$. In fact, transitivity of the order is established for
$\Pm$. 
In Section~\ref{s:weird}, we discuss partitions of the identity
and show that transitivity fails for
$\Fm$. 
In Sections~\ref{s:Z}, \ref{s:L}, and \ref{s:M}, we connect the
resolvent order to the orders by
Zarantonello, by Loewner, and by Moreau, respectively. 
New orders on $\Mm$ and $\Gx$ are introduced in
Section~\ref{s:no}. 
These orders are not \emph{partial orders}.
A quotient construction is presented in 
Section~\ref{s:last} which results in partial orders on $\Pm$ and
on $\Gx$. 

The notation we employ is standard and follows, e.g.,
\cite{BC2011}. 

\section{Auxiliary results}

\label{s:aux}

In this section, we collect various results that will be useful
later. 

\begin{fact}
\label{f:fne}
Let $T\colon X\to X$. Then the following are equivalent:
\begin{enumerate}
\item
$T$ is \emph{firmly nonexpansive}, i.e.,
$(\forall x\in X)(y\in X)$
$\scal{Tx-Ty}{x-y}\geq\|Tx-Ty\|^2$.
\item $(\forall x\in X)(y\in X)$
$\|Tx-Ty\|^2 + \|(x-Tx)-(y-Ty)\|^2\leq\|x-y\|^2$.
\item 
$\Id-T$ is firmly nonexpansive.
\item
$2T-\Id$ is nonexpansive. 
\end{enumerate}
\end{fact}
\begin{proof}
See, e.g., \cite{BC2011}, \cite{GK}, or \cite{GR}.
\end{proof}

\begin{corollary}
\label{c:bibi}
The sets $\Fm$ and $\Pm$ are convex. 
If $\lambda\in[0,1]$, then
$\lambda\Fm\subseteq\Fm$ and $\lambda\Pm\subseteq\Pm$.
\end{corollary}
\begin{proof}
For $\Fm$, the convexity follows using the last item from Fact~\ref{f:fne}
(see also \cite[Corollary~1.8]{BBMW}). 
For the convexity of $\Pm$, see \cite{Moreau}. 
To obtain the inclusions, it suffices to note that $0\in\Pm$. 
\end{proof}

\begin{lemma}
\label{l:160618a}
Let $T_1$ and $T_2$ be firmly nonexpansive on $X$.
Then $T_2-T_1$ is nonexpansive.
\end{lemma}
\begin{proof}
By Fact~\ref{f:fne},
we can write each 
$T_i = (\Id+N_i)/2$, where $N_i$ is nonexpansive.
It follows that $T_2-T_1 = (N_2-N_1)/2$ is nonexpansive.
\end{proof}

\begin{fact}
\label{f:BH}
Let $f\in\Gx$ and set $h := f^*-\qq$.
Then the following are equivalent:
\begin{enumerate}
\item
\label{f:BHi}
$f$ is \frechet\ differentiable on $X$ and $\nabla f$ is
nonexpansive.
\item
\label{f:BHii}
$f$ is \frechet\ differentiable on $X$ and $\nabla f$ is
firmly nonexpansive.
\item
\label{f:BHiii}
$\qq-f$ is convex
\item
\label{f:BHiv}
$f^*-\qq$ is convex. 
\item
\label{f:BHv}
$h\in\Gx$ and $f=\env{h^*}$.
\item
\label{f:BHvi}
$h\in\Gx$ and $\nabla f = \px{h}$. 
\end{enumerate}
\end{fact}
\begin{proof}
See \cite[Theorem~18.17]{BC2011}, 
and also \cite{BH}, \cite{BC2010}, 
\cite{Moreau}.
\end{proof}

\begin{corollary}
\label{c:160618b}
Any linear combination of proximal mappings that is monotone and
nonexpansive is actually a proximal mapping.
\end{corollary}
\begin{proof}
Let $h_1,\ldots,h_n$ be in $\Gx$ such that
$P := \sum_{i=1}^n \alpha_i\px{h_i}$ is nonexpansive and monotone. 
Set $(\forall i\in\{1,\ldots,n\})$ $f_i := \env{h_i^*} = h_i^*\infconv\qq =
(h_i+\qq)^*$,
$f:=\sum_{i=1}^n\alpha_i f_i$ and $h := f^*-\qq$.
Then $P=\nabla f$ and $f$ is thus convex and \frechet\
differentiable. 
By Fact~\ref{f:BH}, $h\in\Gx$ and $P = \nabla f = \px{h}$.
\end{proof}

\begin{corollary}
\label{c:magda}
Let $P_1$ and $P_2$ be proximal mappings.
Then $P_1-P_2$ is a proximal mapping if and only if 
$P_1-P_2$ is monotone.
\end{corollary}
\begin{proof}
``$\Rightarrow$'': Clear.
``$\Leftarrow$'': By Lemma~\ref{l:160618a}, $P_1-P_2$ is
nonexpansive. Now apply Corollary~\ref{c:160618b}.
\end{proof}

\begin{fact}
\label{f:BHlin}
Let $T\colon X\to X$ be linear and self-adjoint.
Then $T$ is firmly nonexpansive
if and only if $T$ is monotone and nonexpansive,
in which case $T$ is a proximal mapping
\end{fact}
\begin{proof}
See \cite[Corollary~18.15]{BC2011} and also \cite{BH}. 
\end{proof}

\section{The resolvent order}

\label{s:main}

From the point of view of monotone operator theory,
the set of firmly nonexpansive mappings is the same as the
set of resolvents. This motivates the language in the following
definition.

\begin{definition}
{\rm \textbf{(resolvent order)}}
\label{d:ourorder}
We define on $\Fm$ a binary relation via 
\begin{equation}
% (\forall T_1\in \Fm)(\forall T_2\in\Fm)\quad 
T_1\preceq T_2
\;:\Leftrightarrow\;
T_2-T_1\in\Fm.
\end{equation}
%We also write $T_2\succeq T_1$ instead of $T_1\preceq T_2$.
\end{definition}

Let us collect some basic properties.

\begin{lemma}
\label{l:160618b}
Let $T,T_0,T_1$ be in $\Fm$.
The binary relation $\preceq$ satisfies the following:
\begin{enumerate}
\item
\label{l:160618bi}
{\rm \textbf{(reflexivity)}}
$T\preceq T$.
\item
\label{l:160618bii}
{\rm \textbf{(existence of least and greatest element)}}
$0\preceq T\preceq \Id$.
\item
\label{l:160618biii}
{\rm \textbf{(order reversal)}}
$T_0\preceq T_1$
$\Leftrightarrow$
$\Id-T_1\preceq \Id-T_0$. 
\item
\label{l:160618biv}
$T_0\preceq T_1$
$\Leftrightarrow$
$(\forall \lambda\in[0,1])$
$T_0 \preceq (1-\lambda)T_0+\lambda T_1\preceq T_1$.
\end{enumerate}
\end{lemma}
\begin{proof}
\ref{l:160618bi}: 
$T-T = 0$ is firmly nonexpansive. 
\ref{l:160618bii}: 
$T-0=T$ is firmly nonexpansive as is
$\Id-T$ by Fact~\ref{f:fne}. 
\ref{l:160618biii}: 
Indeed, $T_1-T_0 = (\Id-T_0)-(\Id-T_1)$.
\ref{l:160618biv}: 
Suppose that $T_0\preceq T_1$, i.e., $T_1-T_0\in\Fm$. 
Let $\lambda\in[0,1]$.
Then $T_1-((1-\lambda)T_0+\lambda T_1) = (1-\lambda)(T_1-T_0)$
and $(1-\lambda)T_0+\lambda T_1-T_0 = \lambda(T_1-T_0)$ both of
which are firmly nonexpansive by Corollary~\ref{c:bibi}. 
The converse implication is trivial. 
\end{proof}

The following two observations are easily verified. 

\begin{example} {\rm \textbf{(lack of symmetry)}}
Suppose that $X\neq\{0\}$. 
Then $0\preceq \Id$ but $\Id\not\preceq 0$.
\end{example}

\begin{example} {\rm \textbf{(lack of antisymmetry)}}
\label{ex:noantisymm}
Suppose that $x_1$ and $x_2$ are two distinct vectors in $X$,
and set $(\forall x\in X)$ $T_1(x) := x_1$ and 
$T_2(x) := x_2$.
Then $T_1\preceq T_2$ and $T_2\preceq T_1$ yet
$T_1\neq T_2$. 
\end{example}

%\begin{remark}
In Section~\ref{s:last}, we will present a quotient construction that makes
the binary relation antisymmetric.
%\end{remark}

We now turn to proximal mappings which allows us to obtain stronger
conclusions. 

\begin{theorem}
\label{t:magda}
Let $f$ and $g$ be in $\Gx$.
Then the following are equivalent:
\begin{enumerate}
\item
\label{t:magdai}
$\px{f}\preceq\px{g}$, i.e., $\px{g}-\px{f}\in\Fm$.
\item 
\label{t:magdaii}
$\px{g}-\px{f}\in\Pm$.
\item
\label{t:magdaiii}
$\env{g^*}-\env{f^*}\in\Ex$. 
\item
\label{t:magdaiv}
$\env{f}-\env{g}\in\Ex$. 
\end{enumerate}
\end{theorem}
\begin{proof}
``\ref{t:magdai}$\Leftarrow$\ref{t:magdaii}'': Clear.
``\ref{t:magdai}$\Rightarrow$\ref{t:magdaii}'': 
Since $\px{g}-\px{f}$ is firmly nonexpansive it is also
monotone. Now apply Corollary~\ref{c:magda}.
``\ref{t:magdaii}$\Rightarrow$\ref{t:magdaiii}'': 
Integrate.
``\ref{t:magdaii}$\Leftarrow$\ref{t:magdaiii}'': 
Differentiate. 
``\ref{t:magdaiii}$\Leftarrow$\ref{t:magdaiv}'': 
This is clear since
$\env{g^*}-\env{f^*} = (\qq-\env{g})-(\qq-\env{f}) =
\env{f}-\env{g}$. 
\end{proof}

\begin{theorem} {\rm \textbf{(transitivity for proximal
mappings)}}
\label{t:proxtransi}
Let $f,g,h$ be in $\Gx$ such that 
$\px{f}\preceq\px{g}$ and
$\px{g}\preceq\px{h}$.
Then $\px{f}\preceq\px{h}$.
\end{theorem}
\begin{proof}
By the hypothesis and Theorem~\ref{t:magda}, 
there exist $a$ and $b$ in $\Gx$ such that
\begin{equation}
\px{g}-\px{f}=\px{a}
\;\;\text{and}\;\;
\px{h}-\px{g}=\px{b}.
\end{equation}
Adding yields $\px{h}-\px{f}=\px{a}+\px{b}$. 
On the other hand, $\px{a}+\px{b}$ is monotone because
$\px{a}$ and $\px{b}$ are monotone. 
Altogether, we deduce from Corollary~\ref{c:magda}, 
that $\px{h}-\px{f}$ is a proximal mapping. 
\end{proof}

\begin{corollary}
{\rm \textbf{(proximal mappings are directed)}}
$(\Pm,\preceq)$ is a directed set.
\end{corollary}
\begin{proof}
The reflexivity of $\preceq$ was observed in
Lemma~\ref{l:160618b}\ref{l:160618bi} while
the transitivity of $\preceq$ 
is a consequence of Theorem~\ref{t:proxtransi}.
Finally, if $P_1$ and $P_2$ are in $\Pm$,
then $P_1\preceq \Id$ and $P_2\preceq \Id$
by Lemma~\ref{l:160618b}\ref{l:160618bii}. 
\end{proof}

We conclude this section with an example.

\begin{example}
Denote the unit ball centered at $0$ of radius $1$ in $X$ by $C$,
and set $T := \Id-\pj{C}$.
Then $(\forall \nnn)$
$T^n = \Id- \pj{nC}\in\Pm$ and
$T^{n}-T^{n+1} = \pj{(n+1)C}-\pj{n C}\in\Pm$. 
Consequently,
\begin{equation}
(\forall\nnn)\quad
0 \preceq T^{n+1}\preceq T^n \preceq \cdots \preceq T \preceq
T^0=\Id. 
\end{equation}
\end{example}
\begin{proof}
The identity for $T^n$ is easily verified by mathematical
induction and discussing cases.
To verify that $T^n-T^{n+1}$, observe that
by Corollary~\ref{c:magda} it suffices to show that 
$T^n-T^{n+1}$ is monotone.
In turn, this is achieved by discussing cases and invoking
the Cauchy--Schwarz inequality. 
\end{proof}

\section{Partitions of the identity and the partial sum property}

\label{s:weird}

In this section, we discuss 
partial sums of firmly nonexpansive mappings arising
in partitions of the identity. 
Somewhat surprisingly, we also show 
that the transitivity result for proximal mappings
(Theorem~\ref{t:proxtransi}) \emph{fails} for firmly nonexpansive
mappings (see Example~\ref{ex:badboy} below). 

We start with a positive result. 

\begin{lemma}
\label{l:three}
Let $T_1,T_2,T_3$ be in $\Fm$ such that
such that $T_1+T_2+T_3=\Id$.
Then $T_1+T_2$ is firmly nonexpansive.
\end{lemma}
\begin{proof}
Since $T_1+T_2=\Id-T_3$, this follows from Fact~\ref{f:fne}. 
\end{proof}

For proximal mappings we are able to extend Lemma~\ref{l:three} 
from $3$ to any number of operators: 

\begin{theorem}
{\rm \textbf{(partial sum property for proximal mappings)}}
\label{t:partialsum}
Let $n\in\{1,2,\ldots\}$, 
let $P_1,\ldots,P_n$ be in $\Pm$ such that
$P_1+P_2+\cdots + P_{n}=\Id$, and
let $m\in\{1,\ldots,n\}$.
Then $P_1+\cdots+P_m\in\Pm$. 
\end{theorem}
\begin{proof}
There exist functions
$f_1,\ldots,f_{n}$ in $\Gx$ such that
for each $i$, 
$\nabla \env{f_i^*} = P_i$, and
$\env{f_1^*}+\cdots+\env{f_{n}^*} = \qq$.
It follows that
\begin{equation}
\qq - \big(\env{f_1^*}+\cdots+\env{f_{m}^*}\big) = 
\env{f_{m+1}^*}+\cdots+\env{f_{n}^*}
\end{equation}
is convex. 
By Fact~\ref{f:BH}, 
$\nabla (\env{f_1^*}+\cdots+\env{f_{m}^*})=P_1+\cdots+P_m$ is a
proximal mapping.
\end{proof}

Surprisingly, the counterpart of Theorem~\ref{t:partialsum}
for firmly nonexpansive mappings is false as the next two results
show.

\begin{lemma}
\label{l:ttr}
In $X=\RR^2$, 
let $n\in\{2,3,\ldots\}$, 
let $\theta \in
\;\big]\negthinspace\arccos(1/\sqrt{2}),\arccos(1/\sqrt{2n})\big]$,
%there exist $\theta\in\left]\pi/4,\pi/2\right[$
set $\alpha:=1/(2n\cos(\theta))$, and 
denote by $R_\theta$ be the counterclockwise rotator by $\theta$.
Then the following hold:
\begin{enumerate}
\item $\alpha R_\theta$ and $\alpha R_{-\theta}$ are firmly
nonexpansive.
\item $n\alpha R_\theta$ and $n\alpha R_{-\theta}$ are not
firmly nonexpansive.
\item $n\alpha R_\theta + n\alpha R_{-\theta} = \Id$.
\end{enumerate}
\end{lemma}
\begin{proof}
Observe that $(\forall x\in X)$
$\scal{x}{R_\theta x} = \cos(\theta)\|x\|^2 =
\cos(\theta)\|R_\theta x\|^2$.
%If $\alpha>0$, then $\scal{x}{\alpha R_\theta
%x}=(\cos(\theta)/\alpha)\|\alpha R_\theta x\|^2$. 
It follows that 
\begin{equation}
\label{e:ttr1}
(\forall \alpha\in\RP)\quad
\alpha R_\theta\in\Fm
\quad\Leftrightarrow\quad
\alpha \in [0,\cos(\theta)].
\end{equation}
%Now let $n\in\{1,2,\ldots\}$ and 
Let $\alpha\in\RPP$ satisfy
\begin{equation}
\label{e:ttr2}
n\alpha R_\theta + n\alpha R_{-\theta} =
n\alpha(R_\theta+R_{-\theta}) = n\alpha\begin{pmatrix}
2\cos(\theta) & 0 \\
0 & 2\cos(\theta)
\end{pmatrix} = \Id;
\end{equation}
equivalently, 
\begin{equation}
\label{e:ttr3}
\alpha := \frac{1}{2n\cos(\theta)}.
\end{equation}
Combining \eqref{e:ttr1} and \eqref{e:ttr3} yields
\begin{equation}
\label{e:ttr4}
\alpha R_\theta\in\Fm
\quad\Leftrightarrow\quad
\frac{1}{2n}\leq \cos^2(\theta).
\end{equation}
On the other hand, 
$(n\alpha) R_\theta\notin\Fm$ 
$\Leftrightarrow$ $n\alpha >\cos(\theta)$ 
$\Leftrightarrow$ $(2\cos(\theta))^{-1} >\cos(\theta)$. 
Altogether,
\begin{subequations}
\label{e:ttr5}
\begin{align}
\big[\alpha R_\theta\in\Fm\;\text{and}\;n\alpha
R_\theta\notin\Fm\big]
&\Leftrightarrow
\cos(\theta) < \frac{1}{2\cos(\theta)} \leq n\cos(\theta)\\
&\Leftrightarrow
\cos^2(\theta) < \frac{1}{2} \leq n\cos^2(\theta)\\
&\Leftrightarrow
\cos(\theta) < \frac{1}{\sqrt{2}} \leq \sqrt{n}\cos(\theta)\\
&\Leftrightarrow
\frac{1}{\sqrt{2n}} \leq \cos(\theta) < \frac{1}{\sqrt{2}}. 
\end{align}
\end{subequations}
Note that \eqref{e:ttr5} has no solution for $n=1$;
however, \eqref{e:ttr5} has solutions for every $n\geq 2$.
Because $R_{-\theta}=R_\theta^*$, the result follows with
\cite[Corollary~4.3]{BC2011} or by arguing along the same lines
as above for $R_{-\theta}$. 
\end{proof}

We now obtain the following direct consequence of
Lemma~\ref{l:ttr}:

\begin{example}
{\rm \textbf{(partial sum property fails for general firmly nonexpansive
mappings)}}
\label{ex:ttr}
Let $n\in\{2,3,\ldots\}$, and let $\theta$, $\alpha$, and
$R_{\pm\theta}$ be as in Lemma~\ref{l:ttr}.
Furthermore, set $T_1 := \cdots = T_n := \alpha R_\theta$ and 
$T_{n+1} := \cdots = T_{2n} := \alpha R_{-\theta}$. 
Then each $T_i$ is firmly nonexpansive, 
$T_1+\cdots +T_{2n}=\Id$, yet
$T_1+\cdots+T_n$ is \emph{not} firmly nonexpansive.
\end{example}
We conclude this section with another negative result.

\begin{example}
\label{ex:badboy}
{\rm \textbf{(lack of transitivity for firmly nonexpansive
mappings})}
Suppose that $X=\RR^2$, and set 
$R:=\alpha R_\theta$ and $S := \alpha R_{-\theta}$,
where $\theta$ and $\alpha$ are as in Example~\ref{ex:ttr} for
$n=2$.
Then 
\begin{subequations}
\begin{align}
\label{e:fd1}
&\text{$R$ and $S$ are firmly nonexpansive,}\\
\label{e:fd2}
&\text{$2R$ and $2S$ are \emph{not} firmly nonexpansive,}\\
\label{e:fd3}
&\text{$2R+2S = \Id$.}
\end{align}
\end{subequations}
Now set
\begin{equation}
T_1 := S,\;\;
T_2 := R+S,\;\;
T_3 := 2R+S.
\end{equation}
Then $T_1\in\Fm$ by \eqref{e:fd1}.
Next, \eqref{e:fd1} and \eqref{e:fd3} imply
that $T_3 = \Id-S\in\Fm$. 
Since $\Fm$ is convex (Corollary~\ref{c:bibi}), it follows that
$T_2 = (T_1+T_3)/2 \in \Fm$.
Because $T_2-T_1 =T_3-T_2= R\in\Fm$  by \eqref{e:fd1}, we have
\begin{equation}
T_1\preceq T_2
\;\;\text{and}\;\;
T_2\preceq T_3.
\end{equation}
On the other hand,
$T_3-T_1 = 2R\notin\Fm$ by \eqref{e:fd2}.
Thus,
\begin{equation}
T_1\not\preceq T_3.
\end{equation}
Altogether, we deduce that 
\begin{equation}
\text{
$(\Fm,\preceq)$ is \emph{not}
transitive.
}
\end{equation}
\end{example}

\section{Compatibility with Zarantonello's partial order}

\label{s:Z}

Zarantonello introduced in \cite{Zara1,Zara2}
a partial ordering of the set of projectors
onto nonempty closed convex cones contained in $X$ via
\begin{equation}
\label{e:Z}
\pj{C}\preceqz \pj{D}
\;\;:\Leftrightarrow\;\; \pj{C}\pj{D} = \pj{C}.
\end{equation}
He established various nice properties which we
collect in the following result.

\begin{fact}
{\rm \textbf{(Zarantonello)}}
\label{f:Z}
Let $C$ and $D$ be nonempty closed convex cones in $X$.
Then the following hold:
\begin{enumerate}
\item
\label{f:Zi}
$\pj{C} \preceqz \pj{D}$ 
$\Leftrightarrow$
$\pj{D}-\pj{C}$ is a projector, in which case\footnote{Here
$C^\ominus := \menge{x\in X}{\sup\scal{C}{x}=0}$ is the \emph{polar
cone} of $C$.}
$\pj{D}-\pj{C}=\pj{D\cap C^\ominus}$.
\item
\label{f:Zii}
$\pj{C} \preceqz \pj{D}$ 
$\Leftrightarrow$
$[\pj{C}\pj{D} = \pj{D}\pj{C}$ and
$(\forall x\in X)$ $\scal{x}{\pj{C}x}\leq\scal{x}{\pj{D}x}]$.
\item
\label{f:Ziii}
$\pj{C} \preceqz \pj{D}$ 
$\Rightarrow$ $C\subseteq D$.
\item
\label{f:Ziv}
$\pj{C} \preceqz \pj{D}$ 
$\Rightarrow$ 
$\pj{C},\pj{D},\pj{C^\ominus},\pj{D^\ominus}$ pairwise commute
with their products being the projectors onto the
intersection of their ranges 
($\pj{C}\pj{D}=\pj{C}\pj{D}=\pj{C\cap D}$, etc.).
\item
\label{f:Zv}
Suppose that $C$ and $D$ are subspaces.
Then
$\pj{C} \preceqz \pj{D}$ 
$\Leftrightarrow$
$C\subseteq D$.
\end{enumerate}
\end{fact}
\begin{proof}
\ref{f:Zi}--\ref{f:Ziv}: 
See \cite[Lemma~5.12]{Zara1} and \cite[page~347]{Zara2}. 
\ref{f:Zv}:
If $C\subseteq D$, then $D = C \oplus (D\cap C^\perp)$,
which implies that $\pj{D}-\pj{C} = \pj{D\cap C^\perp}$ is a projector
and $\pj{C}\preceqz\pj{D}$ by \ref{f:Zi}.
The other implication is \ref{f:Ziii}.
\end{proof}

\begin{remark}
Generalizing 
\eqref{e:Z} and hoping that Fact~\ref{f:Z}\ref{f:Zi} holds 
by just replacing projectors by proximal mappings will not work:
indeed, 
$\px{f}-\px{f} =0 = \pj{\{0\}} $ is a proximal map and a
projector yet $\px{f}\px{f}\neq\px{f}$. 
\end{remark}

Next, let us show that Zarantonello's order is compatible with the
order from Definition~\ref{d:ourorder}:

\begin{lemma}
{\rm \textbf{(compatibility with Zarantonello's order)}}
\label{l:Zcomp}
Let $C$ and $D$ be nonempty closed convex cones in $X$.
Then $\pj{C}\preceqz\pj{D}$ 
$\Leftrightarrow$
$\pj{C}\preceq\pj{D}$. 
\end{lemma}
\begin{proof}
``$\Rightarrow$'': 
Assume that $\pj{C}\preceqz\pj{D}$. 
By Fact~\ref{f:Z}\ref{f:Zi}, $\pj{D}-\pj{C}$ is a projector,
hence a proximal mapping and thus firmly nonexpansive.
Therefore, $\pj{C}\preceq\pj{D}$. 
``$\Leftarrow$'':
Assume that $\pj{C}\preceq\pj{D}$, i.e., 
$\pj{D}-\pj{C}$ is firmly nonexpansive.
By Theorem~\ref{t:magda},
$\pj{D}-\pj{C}$ is a proximal mapping. 
Hence, for 
\begin{equation}
f := \thalb d_C^2 - \thalb d_D^2,
\end{equation}
there exists $h\in\Gx$ such that 
\begin{equation}
\label{e:fdz}
\nabla f = (\Id-\pj{C})-(\Id-\pj{D}) = \pj{D}-\pj{C} = \px{h}.
\end{equation}
Since $\px{h}$ is monotone, the function $f$ is convex and so
$f\in\Gx$. 
Now $\nabla f = \px{h}=(\Id+\sd{h})^{-1}$
$\Rightarrow$
$\sd{f^*}=(\nabla f)^{-1} =\Id+\sd{h} = \sd{(\qq+h)}$.
Thus, after integrating and noting that \eqref{e:fdz} is
invariant under adding constants to $h$, we may and do assume that 
\begin{equation}
h = f^* - \qq.
\end{equation}
Furthermore, combining 
\cite[Example~13.3(ii) and Example~13.24(iii)]{BC2011}
yields 
\begin{equation}
\big(\thalb d_C^2\big)^* = \qq + \iota_{C^\ominus}.
\end{equation}
Let us now compute $f^*$ at $u\in X$.
By \cite[Proposition~14.19]{BC2011}, 
\begin{subequations}
\begin{align}
f^*(u) &= \sup_{v\in \dom ((1/2)
d_D^2)^* }\Big( \big(\thalb d_C^2\big)^*(u+v)- \big(\thalb
d_D^2\big)^*(v)\Big)\\
&=  \sup_{v\in D^\ominus}\big( \qq(u+v)+\iota_{C^\ominus}(u+v)-
\qq(v)\big)\\
&= \qq(u)+ \sup_{v\in D^\ominus}\big(
\scal{u}{v}+\iota_{C^\ominus}(u+v) \big).
\end{align}
Two cases are now conceivable.
\end{subequations}

\emph{Case~1:} $(\exi v\in D^\ominus)$ $u+v\notin C^\ominus$.\\
Then $f^*(u)=\pinf$ and hence $f^*(u)-q(u)=\pinf$.

\emph{Case~2:} $u+D^\ominus \subseteq C^\ominus$.\\
Then 
\begin{equation}
f^*(u)-q(u) = \sup_{v\in D^\ominus}\scal{u}{v} =
\iota_{D^\ominus}^*(u) = \iota_{D^{\ominus\ominus}}(u) =
\iota_D(u)\in\{0,\pinf\}.
\end{equation}

Altogether, $h = f^*-q$ takes only values in $\{0,\pinf\}$, i.e.,
$h$ is an \emph{indicator function}
Thus, $\px{h}$ must be a projector\footnote{We may obtain
additional information as follows. Suppose first, as in
\emph{Case~2}, that $u+D^\ominus\subseteq C^\ominus$.
This case must occur since $h$ is proper.
(In passing, note that this precisely states that $u$ is in the so-called
\emph{star-difference} $C^\ominus \stardiff D^\ominus$; see \cite{HU1}.)
Since $0\in D^\ominus$, it is clear
that $u\in C^\ominus$.
On the other hand, since we are working with \emph{cones}, we have
$(\forall\varepsilon>0)$ $\varepsilon u + D^\ominus
=\varepsilon(u+D^\ominus)\subseteq \varepsilon C^\ominus = C^\ominus$.
Letting $\varepsilon\to 0^+$, we deduce that $D^\ominus \subseteq
C^\ominus$ and thus $C\subseteq D$ as is also guaranteed by 
Fact~\ref{f:Z}\ref{f:Zi}.
If conversely $u\in C^\ominus$, then $u+D^\ominus \subseteq
u+C^\ominus\subseteq C^\ominus$.
Altogether, we have shown that $u$ is as in \emph{Case~2} if and
only if $u\in C^\ominus$. 
Therefore, 
$h=\iota_{D\cap C^\ominus}$, which 
is consistent with Fact~\ref{f:Z}\ref{f:Zi}.}. 
Therefore, by Fact~\ref{f:Z}\ref{f:Zi},
$\pj{C}\preceqz\pj{D}$. 
\end{proof}

\section{Compatibility with the Loewner order via resolvents}

\label{s:L}

In this section, we assume that 
\begin{equation}
X = \SR^n := \menge{A\in\RR^{n\times n}}{A=A^*}
\end{equation}
is the finite-dimensional Hilbert 
space of all real symmetric matrices of size $n\times n$ with the
inner product $\scal{A}{B}$ being the trace of $AB$.
We shall focus on the closed convex cone of 
\emph{positive semidefinite matrices}:
\begin{equation}
\SR^n_+ := \menge{A\in \SR^n}{(\forall x\in\RR^n)\;\scal{x}{Ax}\geq 0}
= \menge{A\in \SR^n}{\text{$A$ is monotone}}. 
\end{equation}
Let $A$ and $B$ be in $\SR^n_+$. 
The classical \emph{Loewner} (or L\"owner) \emph{order} \cite{Lowner} 
states
\begin{equation}
B \preceql A
\;\;:\Leftrightarrow\;\;
A-B\in\SR^n_+, \;\text{i.e., $A-B$ is monotone.}
\end{equation}
Passing to resolvents, we have 
\begin{equation}
\label{e:passing}
B \preceql A
\;\;\Leftrightarrow\;\;
\Id+B \preceql \Id+A
\;\;\Leftrightarrow\;\;
(\Id+A)^{-1} \preceql (\Id+B)^{-1}
\;\;\Leftrightarrow\;\;
\jx{A} \preceql \jx{B}.
\end{equation}
The question now arises whether the Loewner order for
resolvents is compatible with our order from
Definition~\ref{d:ourorder}.
Clearly, 
\begin{equation}
\jx{A}\preceq\jx{B}
\;\;\Rightarrow\;\;
\text{$\jx{B}-\jx{A}$ is monotone}
\;\;\Leftrightarrow\;\;
\jx{A} \preceql \jx{B}.
\end{equation}
Conversely, assume that 
$\jx{A} \preceql \jx{B}$, i.e.,
$\jx{B}-\jx{A}$ is monotone. 
On the other hand,
$\jx{B}-\jx{A}$ is nonexpansive
by Lemma~\ref{l:160618a}. 
Altogether, by Fact~\ref{f:BHlin},
$\jx{B}-\jx{A}$ is firmly nonexpansive, i.e.,
$\jx{A}\preceq\jx{B}$. 
In summary, 
\begin{equation}
\label{e:Lownersummary}
B\preceql A
\;\;\Leftrightarrow\;\;
\jx{A}\preceql\jx{B}
\;\;\Leftrightarrow\;\;
\jx{A} \preceq \jx{B},
\end{equation}
which shows that \emph{the Loewner order and our order are
compatible}.
(In passing, we note that the comments in this section 
have extensions to self-adjoint operators on Hilbert space.)

\section{A connection to Moreau's order}

\label{s:M}

In his seminal work \cite{Moreau}, 
Moreau introduced an order of $\Gx$ via
\begin{equation}
g\preceqm f
\quad:\Leftrightarrow\quad
(\exi h\in\Gx)\;\; f = g+h.
\end{equation}
In fact, using this notation, we can write
the equivalence of 
\ref{f:BHiii} and \ref{f:BHiv} in
Fact~\ref{f:BH},
which was first observed by Moreau \cite{Moreau},
more succinctly as
$f\preceqm\qq$
$\Leftrightarrow$
$\qq\preceqm f^*$.
To make the connection with our order, 
let us take $f$ and $g$ from $\Gx$. 
Using Corollary~\ref{c:magda} and 
Theorem~\ref{t:magda}, 
we  have the following equivalences:
\begin{subequations}
\label{e:Moreauequiv}
\begin{align}
\env{g}\preceqm\env{f}
&\Leftrightarrow
\env{f}-\env{g}\;\text{is convex}
\Leftrightarrow
\env{g^*}-\env{f^*}\;\text{is convex}\\
&\Leftrightarrow
\px{g}-\px{f}\;\text{is monotone}
\Leftrightarrow
\env{f}-\env{g}\in\Ex\\
&\Leftrightarrow
\px{g}-\px{f}\in\Pm
\Leftrightarrow
\px{f}\preceq\px{g}.
\end{align}
\end{subequations}
Hence, \emph{our order is compatible with Moreau's order} when
restricted to $\Ex$, the set of envelope functions on $X$.

\section{Ordering monotone operators and convex functions}

\label{s:no}

The classical bijection between 
the maximally monotone operators on $X$ 
and the firmly nonexpansive mappings on $X$,
introduced by Minty \cite{Minty} (see also \cite{EckBer}), is 
\begin{equation}
\label{e:EckBer}
\Mm \to \Fm\colon A \mapsto \jx{A} = (\Id+A)^{-1}.
\end{equation}
We thus define a new binary relation
on $\Mm$ via 
\begin{equation}
B \preceq A
\;\;
:\Leftrightarrow
\;\;
\jx{A}\preceq \jx{B}.
\end{equation}
For instance, we have 
$\nc{X}=0\preceq A \preceq \nc{\{0\}}$ by
Lemma~\ref{l:160618b}\ref{l:160618bii}  while 
Lemma~\ref{l:160618b}\ref{l:160618biv} gives a statement relating
to the resolvent average introduced in \cite{BBMW}. 
The results in Section~\ref{s:L} show that
when $X=\RR^n$ and we consider $\SR^n_+$,
which is a subset of $\Mm$,
then these notions are compatible (see \eqref{e:Lownersummary}). 

Furthermore, we can use this binary relation on $\Mm$ to 
define a new binary relation on $\Gx$ via 
\begin{equation}
g \preceq f
\;\;
:\Leftrightarrow
\;\;
\sd{g}\preceq \sd{f}.
\end{equation}
With these definitions and
using \eqref{e:Moreauequiv}, we have
the equivalences
\begin{equation}
g \preceq f
\;\Leftrightarrow\;
\sd{g}\preceq \sd{f}
\;\Leftrightarrow\;
\jx{\sd{f}}\preceq \jx{\sd{g}}
\;\Leftrightarrow\;
\px{f}\preceq \px{g}
\;\Leftrightarrow\;
\env{g}\preceqm\env{f} 
\end{equation}
which show that our binary relation on $\Gx$  plays along nicely with
Moreau's order. 

Let us present another example. Let $A$ and $B$ be in $\SR^n_+$ and
define the quadratic forms
\begin{equation}
(\forall x\in\RR^n)\quad 
\qq_A(x):=\thalb\scal{x}{Ax}
\;\text{and}\;
\qq_B(x):=\thalb\scal{x}{Bx}.
\end{equation}
Then $\nabla \qq_A = A$ and $\nabla \qq_B=B$.
Using \eqref{e:passing}, we see that 
\begin{subequations}
\begin{align}
\qq_B\leq \qq_A \;\text{(pointwise)}
&\Leftrightarrow
0\leq \qq_{A-B}
\Leftrightarrow
A-B\in\SR^n_+
\Leftrightarrow
B\preceql A
\Leftrightarrow
\jx{A}\preceq \jx{B}\\
&\Leftrightarrow
B\preceq A
\Leftrightarrow
\nabla \qq_B\preceq \nabla \qq_A
\Leftrightarrow
\qq_B\preceq \qq_A
\Leftrightarrow
\px{\qq_A}\preceq \px{\qq_B}\\
&\Leftrightarrow
\env{\qq_B}\preceqm \env{\qq_A}.
\end{align}
\end{subequations}
This nicely illustrates the connections between
the various orders considered in this paper.

\section{New partial orders}

\label{s:last}

In our final section, we introduce a quotient space construction
which remedies the lack of antisymmetry observed in 
Example~\ref{ex:noantisymm}.

We start with a simple but useful result.

\begin{lemma}
\label{l:simple}
Let $T\in\Fm$ be such that $-T\in\Fm$. 
Then $T$ is a constant mapping.
\end{lemma}
\begin{proof}
Take $x$ and $y$ from $X$.
Then 
\begin{equation}
\normsq{Tx-Ty}\leq\scal{Tx-Ty}{x-y}
\;\text{and}\;
\normsq{(-T)x-(-T)y}\leq-\scal{Tx-Ty}{x-y}.
\end{equation}
Thus, 
\begin{subequations}
\begin{align}
0&\leq \normsq{Tx-Ty}\\
&\leq
\min\big\{\scal{Tx-Ty}{x-y},-\scal{Tx-Ty}{x-y}\big\}\\
&= -|\scal{Tx-Ty}{x-y}|\\
&\leq 0.
\end{align}
\end{subequations}
Therefore, $Tx-Ty=0$.
\end{proof}

We now define a binary relation on $\Fm$ via
\begin{equation}
\label{e:simfm}
T_1\sim T_2
\;:\Leftrightarrow\;
T_2-T_1 \;\text{is a constant mapping.}
\end{equation}
It is straightforward to check that 
$(\Fm,\sim)$ is an \emph{equivalence relation}. 
Denote the corresponding \emph{quotient set} by
\begin{equation}
[\Fm] := \Fm/{\negthinspace\sim}.
\end{equation}
Now define a binary relation on $[\Fm]$ by 
\begin{equation}
[T_1]\preceq [T_2] 
\;\; :\Leftrightarrow \;\;
T_1\preceq T_2. 
\end{equation}
Then $([\Fm],\preceq)$ is \emph{reflexive};
moreover, by Lemma~\ref{l:simple},
$([\Fm],\preceq)$ is \emph{antisymmetric}. 
If we restrict to proximal mappings,
then 
$([\Pm],\preceq)$ is also \emph{transitive}
by Theorem~\ref{t:proxtransi}. 
In summary, 
\begin{equation}
\label{e:super}
\text{$\big(\,[\Pm],\preceq\big)$ \emph{is a partially ordered
set}.}
\end{equation}

Finally, let us investigate \eqref{e:simfm}
from the view point of maximally monotone operators via
\eqref{e:EckBer}. 

\begin{lemma}
\label{l:quotient}
Let $A$ and $B$ be in $\Mm$ and let $c\in X$.
Then the following are equivalent:
\begin{enumerate}
\item
\label{l:quotienti}
$(\forall x\in X)$
$\jx{B}x = c + \jx{A}x$. 
\item 
\label{l:quotientii}
$(\forall x\in X)$ $Bx=-c+A(x-c)$.
\item 
\label{l:quotientiii}
$\gr B = (c,-c)+\gr A$. 
\end{enumerate}
\end{lemma}
\begin{proof}
``\ref{l:quotienti}$\Rightarrow$\ref{l:quotientiii}'':
By assumption, $(\forall x\in X)$
$(\jx{B}x,x-\jx{B}x) = (c,-c)+(\jx{A}x,x-\jx{A}x)$.
Using the Minty parametrization (see \cite[(23.18)]{BC2011}) of the graph,
we see that $\gr B = (c,-c)+\gr A$.
The other implications are proved similarly.
\end{proof}

In view of Lemma~\ref{l:quotient},
the equivalence relation \eqref{e:simfm}
in $\Fm$ gives rise to the following 
equivalence relation on $\Mm$:
\begin{equation}
\label{e:simmm}
A\sim B 
\quad:\Leftrightarrow\quad
(\exi c\in X)(\forall x\in X)\;
Bx = -c+A(x-c).
\end{equation}
In turn, we can ``integrate'' \eqref{e:simmm}
to obtain the following equivalence relation on $\Gx$:
\begin{equation}
\label{e:simf}
f\sim g 
\quad:\Leftrightarrow\quad
(\exi c\in X)(\exi \gamma\in\RR)(\forall x\in X)\;
g(x) = f(x-c) - \scal{c}{x}+\gamma.
\end{equation}
The last equivalence relation induces a quotient set
$[\Gx] := \Gx/{\negthinspace\sim}$.
Interpreting \eqref{e:super} in this setting, 
we obtain the following result.
\begin{theorem}
\label{t:shawn}
Equip the quotient set $[\Gx]$ with the binary relation
\begin{equation}
[g] \preceq [f]
\;:\Leftrightarrow\;
[\px{f}]\preceq[\px{g}]
\;\Leftrightarrow\;
\px{f}\preceq \px{g}.
\end{equation}
Then 
\begin{equation}
\text{$\big(\,[\Gx],\preceq\big)$ \emph{is a partially ordered
set}.}
\end{equation}
\end{theorem}

\section*{Acknowledgments}
Sedi Bartz was supported by a postdoctoral fellowship of the Pacific
Institute for the Mathematical Sciences and by NSERC grants of Heinz
Bauschke and Xianfu Wang. Heinz Bauschke was partially supported
by the Canada Research Chair program and by the Natural Sciences
and Engineering Research Council of Canada. 
Xianfu Wang was partially supported by the Natural
Sciences and Engineering Research Council of Canada.

\end{document}